\documentclass[11pt]{amsart}
\usepackage{amsmath,amssymb,amsthm,upref,graphicx,mathrsfs, enumerate}
\usepackage{esint}
\usepackage{textcomp} % provide lots of new symbols
\usepackage{array} % for better arrays (eg matrices) in maths
\usepackage{textcomp} % provide lots of new symbols
\usepackage[
  colorlinks=true,
  linkcolor=blue,
  citecolor=blue,
  urlcolor=blue,
  pagebackref]{hyperref}

\numberwithin{equation}{section}
\allowdisplaybreaks

\newtheorem{theorem}{Theorem}[section]
\newtheorem{prop}[theorem]{Proposition}
\newtheorem{lemma}[theorem]{Lemma}

\theoremstyle{definition}

\newtheorem{remark}[theorem]{Remark}
\theoremstyle{theorem}
\newtheorem{ltheorem}{Theorem}

\newcommand{\Mm}{\mathcal{M}}

\newcommand{\Rn}{\mathbb{R}^n}

\newcommand{\nw}{\nu_{\vec{w}}}

\newcommand{\f}{\frac}
\newcommand{\vc}{\infty}

\def\Xint#1{\mathchoice
{\XXint\displaystyle\textstyle{#1}}%
{\XXint\textstyle\scriptstyle{#1}}%
{\XXint\scriptstyle\scriptscriptstyle{#1}}%
{\XXint\scriptscriptstyle\scriptscriptstyle{#1}}%
\!\int}
\def\XXint#1#2#3{{\setbox0=\hbox{$#1{#2#3}{\int}$}
\vcenter{\hbox{$#2#3$}}\kern-.5\wd0}}

\def\dashint{\Xint-}

\let \a=\alpha

\let \la=\lambda
\let \o=\omega

\let \p=\varphi
\begin{document}

\title[Weighted bounds for multilinear operators with non-smooth kernels]
  {Weighted bounds for multilinear operators with non-smooth kernels}

\authors

\author[T. A. Bui]{The Anh Bui}
\address{The Anh Bui \\
Department of Mathematics, Macquarie University, Ryde 2109 NSW, Australia \&
Department of Mathematics, University of Pedagogy, HoChiMinh City,
Vietnam}
\email{the.bui@mq.edu.au, bt\_anh80@yahoo.com}

\author[J. M. Conde-Alonso]{Jos\'e M. Conde-Alonso}
\address{Jos\'e M. Conde-Alonso \\
Instituto de Ciencias Matem\'aticas, Consejo Superior de Investigaciones Cient\'ificas, C/ Nicol\'as Cabrera 13-15, 28049 Madrid, Spain}
\email{jose.conde@icmat.es}

\author[X. T. Duong]{Xuan Thinh Duong}
\address{Xuan Thinh Duong \\
Department of Mathematics, Macquarie University, Ryde 2109 NSW, Australia}
\email{xuan.duong@mq.edu.au}

\author[M. Hormozi]{Mahdi Hormozi}
\address{Mahdi Hormozi \\
Department of Mathematical Sciences, Division of Mathematics,
University of Gothenburg, Gothenburg 41296, Sweden}
\email{hormozi@chalmers.se}

\subjclass[2010]{Primary: 42B20, 42B25.}
\keywords{Multilinear singular integrals, weighted norm inequalities, Lerner's formual, multilinear Fourier multipliers}

\arraycolsep=1pt

\begin{abstract} Let $T$ be a multilinear {integral} operator which is bounded on certain products of Lebesgue spaces on $\mathbb R^n$. We assume that its associated kernel satisfies some mild regularity condition  which is weaker than the usual H\"older continuity of those in the class of multilinear Calder\'on-Zygmund singular integral operators. In this paper, given a suitable multiple weight $\vec{w}$, we obtain the bound for the weighted norm of multilinear operators $T$ in terms of $\vec{w}$. As applications, we exploit this result to obtain the weighted bounds {for} certain singular integral operators such as linear and multilinear Fourier multipliers and the Riesz transforms associated to Schr\"odinger operators on $\mathbb{R}^n$. Our results are new even {in} the linear case.
\end{abstract}

\maketitle

\section{Introduction}
\par In the past decades, weighted inequalities have been a very attractive realm in harmonic analysis. One basic problem concerning them consists in determining conditions for a given operator to be bounded in $L^p(w)$ with appropriate weights $w$. A sustained research period started with the famous work of Muckenhoupt \cite{Mu} in the seventies. In that work he characterized the class of weights $u,v$ so that the following weak inequality for the Hardy--Littlewood maximal operator $M$ and for $1\leq p < \infty$ holds:
\begin{equation}\label{weak_two_weight_problem}
  ||M(f)||_{L^{p,\infty}(u)} \leq C ||f||_{L^p(v)}.
\end{equation}
This condition on the weights is known as the $A_p$ condition, namely
\begin{equation*}
  [w]_{A_p} := \sup_{Q} \left(\frac{1}{|Q|}\int_Q w(x) dx\right) \left(\frac{1}{|Q|}\int_Q w(x)^{-\frac{1}{p-1}}dx\right)^{p-1} < \infty, \hspace{1em} p>1,
\end{equation*}
where the supremum is taken over all the cubes (or balls) in $\Rn$. For $p>1$, Muckenhoupt proved that the
following strong estimate
\begin{equation*}\label{strong_one_weight_problem}
  \int_{\Rn} (Mf(x))^p w(x) dx \leq C \int_{\Rn} |f(x)|^p w(x) dx, \, f\in L^p(w),
\end{equation*}
holds if and only if $w$ satisfies the $A_p$ condition.
\par After that, harmonic analysts focused their interest on studying weighted inequalities for many different classical operators such as the Hilbert and Riesz transforms and other singular integral operators leading to a vast literature on weighted norm inequalities. However, the classical results did not reflect the quantitative dependence of the $L^p(w)$ operator norm in terms of the relevant constant involving the weights.
The question of the sharp dependence of the norm estimates of a given operator in terms of the $A_p$ constant of the weight was specifically raised by S. Buckley, who proved the following optimal bound for the Hardy--Littlewood operator in \cite{Bu}:
\begin{equation}\label{Buckley_thm}
  \|M\|_{L^p(w)} \leq C_{p} \,[w]_{A_p}^{\frac{1}{p-1}},
\end{equation}
where $C_p$ is a dimensional constant that also depends on $p$, but not on $w$. We say that the estimate in  \eqref{Buckley_thm} is sharp since the exponent $1/(p-1)$ cannot be replaced by a smaller one.
\par On the other hand, it turned out that for singular integral operators the question was much more complicated. The linear bounds for the Hilbert and Riesz transforms were addressed by Petermichl \cite{Pet,Pet2}.
Since then, the so-called $A_2$ conjecture gathered more attention from the mathematical community. This conjecture stated that the sharp dependence of the $L^2(w)$ norm of a Calder\'on--Zygmund operator on the $A_2$ constant of the weight $w$ was linear. Finally, in 2012 T. Hyt\"onen \cite{Hyt1} proved the so-called $A_2$ theorem, which asserted that this was indeed the case. This, in combination with the extrapolation theorem in \cite{DGPP} gives the sharp dependence of the $L^p(w)$ norm for Calder\'on--Zygmund operators with $1<p<\vc$. More precisely, if $T$ is a Calder\'on--Zygmund operator then we have
\begin{equation}
\label{eq:Ap_sharp_CZO}
  \|T\|_{L^p(w)} \leq C_{T,n,p} [w]_{A_p}^{\max{\left(1,\frac{1}{p-1}\right)}}, \,\,  1<p<\infty, w\in A_p.
\end{equation}
Shortly after that, A. K. Lerner gave a much simpler proof \cite{Ler4} of the $A_2$ theorem proving that every Calder\'on--Zygmund operator is bounded from above by a supremum of sparse operators. Namely, if $\mathbb{X}$ is a Banach function space, then
\begin{equation}
\label{eq:domination_CZO}
  \|T(f)\|_{\mathbb{X}} \leq C \sup_{\mathscr{D},\mathcal{S}} \| \mathcal{A}_{\mathscr{D},\mathcal{S}}(f)\|_{\mathbb{X}},
\end{equation}
where the supremum is taken over arbitrary dyadic grids $\mathscr{D}$ and sparse families $\mathcal{S}\subset \mathscr{D}$, and
\begin{equation*}
  \mathcal{A}_{\mathscr{D},\mathcal{S}}(f)= \sum_{Q\in \mathcal{S}}\Big(\dashint_Q f\Big)\chi_Q.
\end{equation*}
The interested readers can consult \cite{Hyt2} for a survey on the history of the proof. The application of Lerner's techniques is reflected in the extension of \eqref{eq:domination_CZO} and the $A_2$ theorem to multilinear Calder\'on--Zygmund operators in \cite{DLP}. Later on, Li, Moen and Sun in \cite{LMS} proved the corresponding sharp weighted $A_{\vec P}$ bounds for multilinear sparse operators. In other words, if $1<p_1,\dotsc,p_m<\infty$ with $\tfrac{1}{p_1}+\dotsb+\tfrac{1}{p_m}=\tfrac{1}{p}$ and $\vec{w}\in A_{\vec{P}}$, then
\begin{equation}
\label{eq:LMS_sparse}
          \|\mathcal{A}_{\mathscr{D},\mathcal{S}}(\vec{f})\|_{L^p(\nu_{\vec{w}})} \lesssim [\vec{w}]_{A_{\vec{P}}}^{\max(1,\tfrac{p_1'}{p},\dotsc,\tfrac{p_m'}{p})} \prod_{i=1}^m \|f_i\|_{L^{p_i}(w_i)},
\end{equation}
for tuples $\vec{f}=(f_1, \ldots, f_m)$. Now the symbol $\mathcal{A}_{\mathscr{D},S}$ denotes the multilinear sparse operator
\begin{equation*}
  \mathcal{A}_{\mathscr{D},S}(\vec{f})(x)=\sum_{Q} \left(\prod_{i=1}^m (f_i)_{Q}\right) \chi_{Q}(x).
\end{equation*}
The readers are referred to \cite{CR, LMS} to observe that from \eqref{eq:LMS_sparse}, we can derive the multilinear $A_{\vec P}$ theorem for $1/m<p<\infty$. More precisely, if $T$ is a multilinear Calder\'on--Zygmund operator, $1<p_1,\dotsc,p_m<\infty$, $\tfrac{1}{p_1}+\dotsb+\tfrac{1}{p_m}=\tfrac{1}{p}$ and $\vec{w}=(w_1,\dotsc,w_m)\in A_{\vec{P}}$, then
  \begin{equation}\label{eq:LMS_CZO}
    \|T(\vec{f})\|_{L^p(\nu_{\vec{w}})} \leq C_{n,m,\vec{P},T} [\vec{w}]_{A_{\vec{P}}}^{\max(1,\tfrac{p_1'}{p},\dotsc,\tfrac{p_m'}{p})} \prod_{i=1}^m \|f_i\|_{L^{p_i}(w_i)}.
  \end{equation}
For further details on the theory of multilinear Calder\'on--Zygmund operators, we refer to \cite{G, GT1} and the references therein.

\par In this paper, we aim to study the weighted bound of certain multilinear singular integral operators on products of weighted Lebesgue spaces. It is important to note that the multilinear singular integral operators considered in our paper {are} beyond the Calder\'on-Zygmund class of multilinear singular integral operators considered in \cite{GT1}. More precisely, in this paper we assume that $T$ is a multilinear operator initially defined on the $m$-fold
product of Schwartz spaces and taking values into the space of
tempered distributions,
\begin{equation*}
T: \mathcal{S}(\mathbb{R}^n)\times\ldots\times
\mathcal{S}(\mathbb{R}^n) \rightarrow \mathcal{S}'(\mathbb{R}^n)
\end{equation*}
By the associated kernel $K(x, y_1, \ldots, y_m)$, we mean that $K$ is a function defined away from the diagonal $x = y_1 =\ldots=y_m$ in
$(\mathbb{R}^n)^{m+1}$, satisfying
\begin{equation*}
T (f_1, \cdots, f_m)(x) = \int_{(\mathbb{R}^n)^m} K(x,y_1,\ldots ,
y_m)f_1(y_1)\ldots f_m(y_m)dy_1 \ldots dy_m
\end{equation*}
for   all functions $f_j\in \mathcal{S}(\mathbb{R}^n)$ and all $x \notin \cap^m_{j=1}$supp$f_j$, $j=1,\ldots, m$.

\medskip

In what follows, we denote $dy_1\ldots dy_m$ by $d\vec{y}$. For the rest of this paper, we assume that there exist $p_0\geq 1$ and a constant $C>0$ so that the following conditions holds:

\begin{enumerate}[\textbf{(H1)}]
\item $T$ maps $L^{p_0}\times\ldots\times L^{p_0}$ into $L^{p_0/m,\vc}$.
\end{enumerate}
\begin{enumerate}[\textbf{(H2)}]
\item There exists $\delta> n/p_0$ so that for the conjugate exponent $p'_0$ of $p_0$, one has
\begin{equation}\label{cond2}
\begin{aligned}
\Big(\int_{S_{j_m}(Q)}\ldots \int_{S_{j_1}(Q)}&|K(x,y_1, \ldots, y_m)-K(\overline{x},y_1,\ldots, y_m)|^{p'_0}d\vec{y}\Big)^{1/{p'_0}}\\
&\leq C\f{|x-\overline{x}|^{m(\delta-n/{p_0})}}{|Q|^{m\delta/n}} 2^{-m\delta j_0}
\end{aligned}
\end{equation}
for all balls $Q$, all $x, \overline{x}\in \f{1}{2}Q$ and $(j_1,\ldots,y_m)\neq (0,\ldots,0)$, where $j_0=\max\{j_k: k=1,\ldots, m\}$ and $S_j(B)=2^{j}Q\backslash 2^{j-1}Q$ if $j\geq 1$, otherwise $S_j(Q)=Q$.
\end{enumerate}
Note that we do not require any regularity condition on the kernel of the multilinear operators $T$. The class of operators satisfying the conditions (H1) and (H2) is motivated from the recent works \cite{KW, BD, GT1, LOPTT, LRT, LMRT, LMPR}. In the linear case as $m=1$, the class of such operators is contained implicitly in \cite{KW}. The condition (H2) is similar to the $L^r$-H\"omander conditions considered in \cite{LRT, LMRT, LMPR}. In the multilinear case, this kind of operators was considered by the first and the third authors in \cite{BD} in studying the weighted norm inequalities of multilinear Fourier multiplier operators with limited smoothness symbols. More importantly, the class of operators satisfying the conditions (H1) and (H2) includes the class of multilinear Calder\'on-Zygmund singular integral operators (see \cite{GT1, LOPTT} for the precise definition). More precisely, if $T$ is a multilinear Calder\'on-Zygmund singular integral operator then it is easy to see tha
 t $T$ satisfies (H1) and (H2) with $p_0=1$.\\

The main goal of this paper is to obtain the weighted bounds for multilinear singular integrals which satisfy (H1) and (H2). According to a standard approach, it is natural to consider the following multi-sublinear operators. Fix $p_0\in [1,\vc)$ and a dyadic grid $\mathscr{D}\subset \mathbb{R}^n$. Define, for any cube $Q$,
$$
\langle f \rangle_{Q,p_0} := \left(\frac{1}{|Q|} \int_Q |f(x)|^{p_0}dx \right)^{\frac{1}{p_0}}.
$$
For $k\geq 0$, denote by $\mathcal{A}^{k,p_0}_{\mathcal{S}}$ the $m$-sublinear sparse operator
\begin{equation*}
  \mathcal{A}^{k,p_0}_{\mathcal{S}}(\vec{f})(x)=\sum_{Q\in S} \left[\prod_{i=1}^m \langle f_i \rangle_{Q^{(k)},p_0}\right] \chi_{Q}(x),
\end{equation*}
which acts on measurable $m$-tuples $\vec{f}=(f_1,\ldots, f_m)$. Here $Q^{(k)}$ denotes the $k$-th dyadic ancestor of $Q$ in $\mathscr{D}$. Also, we may define the operator $\mathcal{T}_{\mathcal{S}}^{k,p_0}$ by
\begin{equation*}
  \mathcal{T}^{k,p_0}_{\mathcal{S}}(\vec{f})(x)=\sum_{Q\in S} \left[\prod_{i=1}^m \langle f_i \rangle_{2^kQ,p_0}\right] \chi_{Q}(x),
\end{equation*}
We shall also work with the localized versions of the operators above, in which we will consider that the sum in the definition ranges over cubes $Q$ contained in some fixed cube $P$. We will denote them respectively by $\mathcal{A}^{k,p_0}_{\mathcal{S},P}$ and $\mathcal{T}^{k,p_0}_{\mathcal{S},P}$. Our first main result reads as follows:

\begin{ltheorem}
\label{theoremA}
Let $\mathbb{X}$ be a quasi Banach function space (in the sense of \cite{CR}). Then, for each sparse family $\mathcal{S}$ there exists another sparse family $\mathcal{S}'$ such that we have
$$
\|\mathcal{T}^{k,p_0}_{\mathcal{S}}(\vec{f})\|_{\mathbb{X}} \leq C(k+1) \|\mathcal{A}^{0,p_0}_{\mathcal{S'}}(\vec{f})\|_{\mathbb{X}},
$$
for some constant $C$ that may depend on $\mathbb{X}$ and $k$, but not on $k$ or $\vec{f}$.
\end{ltheorem}

The proof of theorem \ref{theoremA} follows the scheme of \cite{CR}, in which the case $p_0=1$ is considered. The main new difficulty is the fact that the operator $\mathcal{A}^{m,p_0}_{\mathcal{S}}$ is not linear for $p_0 \not=1$. Of course, $\mathcal{A}^{m,p}_{\mathcal{S}} \vec{f} \geq \mathcal{A}^{m,q}_{\mathcal{S}} \vec{f}$ for positive tuples $\vec{f}$ and and $p\geq q$. Therefore, bounding the operators $\mathcal{A}^{m,p_0}_{\mathcal{S}}$ for $p_0>1$ leaves some space for estimates involving Calder\'on-Zygmund operators with rough kernels. {On the other hand, the operators $\mathcal{A}^{p_0}_{\mathcal{S}} := \mathcal{A}^{0,p_0}_{\mathcal{S}}$ have very nice quantitative properties:}

\begin{ltheorem}
  \label{theoremB}
    Suppose that $p_0<p_1,\ldots,p_m<\infty$ with $\tfrac{1}{p_1}+\ldots+\tfrac{1}{p_m}=\tfrac{1}{p}$ and $\vec{w}\in A_{\vec{P}/p_0}$. Then
        \begin{equation*}
          \| \mathcal{A}^{p_0}_{\mathcal{S}}(\vec{f})\|_{L^p(\nu_{\vec{w}})} \lesssim [\vec{w}]_{A_{\vec{P}/p_0}}^{\max\left(1,\tfrac{(p_1/p_0)'}{p},\dotsc,\tfrac{(p_m/p_0)'}{p}\right)} \prod_{i=1}^m \|f_i\|_{L^{p_i}(w_i)}.
        \end{equation*}
  \end{ltheorem}

{This is all we need to get the weighted bounds of our operators with nonsmooth kernels.}

\begin{ltheorem}
\label{theoremC}
Let $T$ satisfy (H1) and (H2). Suppose that $p_0<p_1,\ldots,p_m<\infty$ with $\tfrac{1}{p_1}+\ldots+\tfrac{1}{p_m}=\tfrac{1}{p}$ and $\vec{w}\in A_{\vec{P}/p_0}$.  Then
        \begin{equation*}
          \| T(\vec{f})\|_{L^p(\nu_{\vec{w}})} \lesssim  [\vec{w}]_{A_{\vec{P}/p_0}}^{\max\left(1,\tfrac{(p_1/p_0)'}{p},\dotsc,\tfrac{(p_m/p_0)'}{p}\right)}
 \prod_{i=1}^m \|f_i\|_{L^{p_i}(w_i)}.
        \end{equation*}
\end{ltheorem}

 We would like to point out that our results are new even for the linear case. Although our conjecture is that these bounds are sharp, we {could not} prove this and leave it as an open problem.

\bigskip

\par The outline of the rest of this paper is the following: In the next section we recall the definition of multiple weights and  Lerner's local oscillation formula. Section \ref{secpointwise} is devoted to prove Theorem \ref{theoremA}. The proof of Theorem \ref{theoremB} and Theorem \ref{theoremC} are given in Section \ref{secapp}. Finally, in Section \ref{App-sec}, we apply the obtained result in Theorem \ref{theoremC} to consider the weighted bounds of certain singular integral operators such as linear and multilinear Fourier multipliers and Riesz transforms associated to Schr\"odinger operators.

\par Throughout the text, $A\lesssim B$ will denote $A\leq C B$, where $C$ will denote a positive constant independent of the weight which may change from {line to line}. Moreover, $A\lesssim_{{a,b}} B$ will denote $A\leq C B$, where $C$ will denote a positive constant dependent on $a$ and $b$.

\section{Preliminaries}
\label{Sect:Preliminaries}
%**********************************************************************************************************%

\subsection{Multiple weight theory}
For a general account on multiple weights and related results we refer the interested reader to \cite{LOPTT}. In this section we briefly introduce some definitions and results {that} we will need. Consider $m$ weights $w_1,\dotsc,w_m$ and denote $\overrightarrow{w}=(w_1,\dotsc,w_m)$.  Also let $1<p_1,\dotsc,p_m<\infty$ and ${0<p<\infty}$ be numbers such that $\frac{1}{p}=\frac{1}{p_1}+\dotsb+\frac{1}{p_m}$, and denote $\overrightarrow P = (p_1,\dotsc, p_m)$. Set

  $$\nu_{\vec w}:=\prod_{i=1}^m w_i^{\frac{p}{p_i}}.$$
We say that $\vec{w}$ satisfies the $A_{\vec{P}}$ condition if
\begin{equation}\label{eq:multiap_LOPTT}
  [\vec{w}]_{A_{\vec{P}}}:=\sup_{Q}\Big(\frac{1}{|Q|}\int_Q\nu_{\vec w}\Big)\prod_{i=1}^m\Big(\frac{1}{|Q|}\int_Q w_i^{1-p'_i}\Big)^{p/p'_i}<\infty.
\end{equation}
When $p_i=1$ {for some $i$}, $\Big(\frac{1}{|Q|}\int_Q w_i^{1-p'_i}\Big)^{p/p'_i}$ is understood as $\displaystyle(\inf_Qw_i)^{-p}$. This condition, introduced in \cite{LOPTT}, was shown to characterize the classes of weights for which the multilinear maximal function $\Mm$ is bounded from $L^{p_1}(w_1)\times\dotsb\times L^{p_m}(w_m)$ into $L^p(\nu_{\vec{w}})$ (see \cite[Thm. 3.7]{LOPTT}). We also denote by $A_p, 1\leq p<\vc$ and $RH_q, 1<q\leq \vc$ the classes of Muckenhoupt weights and the classes of reverse H\"older weights on $\mathbb{R}^n$, respectively. For $w\in A_p, 1\leq p <\vc$, the quantity $[w]_{A_p}$ is defined as \eqref{eq:multiap_LOPTT} by
\begin{equation*}
  [w]_{A_p} := \sup_{Q} \left(\frac{1}{|Q|}\int_Q w(x) dx\right) \left(\frac{1}{|Q|}\int_Q w(x)^{-\frac{1}{p-1}}dx\right)^{p-1},
\end{equation*}
with the usual modification {when} $p=1$. The supremum above is taken over all cubes (or balls) in $\Rn$. For $w\in RH_{q}, 1<q\leq \vc$, we define
\begin{equation*}
  [w]_{RH_q} := \sup_{Q} \left(\frac{1}{|Q|}\int_Q w(x)^q dx\right)^{1/q} \left(\frac{1}{|Q|}\int_Q w(x)dx\right)^{-1},
\end{equation*}
with the usual modification {when} $q=\vc$. Again, the supremum is taken over all cubes (or balls) in $\Rn$.

Let $\sigma\in A_\infty=\cup_{p\geq 1}A_p$. The dyadic maximal function with respect to $\sigma$ is defined as
\begin{equation}\label{bbs}
M_{\sigma}^{\mathscr{D}}(f)(x)= \sup_{  \begin{subarray}{c} x \in Q \\ Q\in  \mathscr{D} \end{subarray} } \frac{1}{\sigma(Q)} \int_{Q} |f|\sigma.
\end{equation}
It is well-known that
\begin{equation}\label{bbs}
\| M_{\sigma}^{\mathscr{D}}f\|_{L^p(\sigma) } \leq p' \| f\|_{L^p(\sigma) },~~~~~~~~~~1<p<\infty.
\end{equation}
See e.g \cite{Moen}.

\subsection{A local mean oscillation formula}

 For the notion of  {\it general dyadic grid} ${\mathscr{D}}$ we refer to previous {works} (e.g. \cite{Ler3} and \cite{Hyt2}). A collection $\mathcal{S}=\{Q\} \subset \mathscr{D}$ is called a {\it sparse family} of cubes if there exist pairwise disjoint subsets $E_Q \subset Q$ with $|Q|\le 2|E_Q|$ for each $Q \in \mathcal{S}$. The major tool to prove our main results is Lerner's local oscillation formula from \cite{Ler3}. To formulate it we need to introduce {several} notions. By a median value of measurable function $f$ on a set $Q$ we mean a possibly nonunique, real number $m_f(Q)$ such that
$$\max\big(|\{x\in Q: f(x)>m_f(Q)\}|,|\{x\in Q: f(x)<m_f(Q)\}|\big)\le |Q|/2.$$

The decreasing rearrangement of a measurable function $f$ on ${\mathbb R}^n$ is defined by
$$f^*(t)=\inf\{\a>0:|\{x\in {\mathbb R}^n:|f(x)|>\a\}|<t\},\quad0<t<\infty.$$
The local mean oscillation of $f$ is
$$\o_{\la}(f;Q)=\inf_{c\in {\mathbb R}}
\big((f-c)\chi_{Q}\big)^*\big(\la|Q|\big),\quad0<\la<1.
$$
Then it follows from the definitions that
\begin{equation}\label{pro1}
|m_f(Q)|\le (f\chi_Q)^*(|Q|/2).
\end{equation}
The following theorem was proved in by Hyt\"onen \cite[Theorem 2.3]{Hyt2} in order to improve original Lerner's formula given in \cite{L1,Ler3}.
\begin{theorem}\label{decom1}
Let $f$ be a measurable function on ${\mathbb R}^n$ and let $Q_0$ be a fixed cube. Then there exists a
(possibly empty) sparse family $\mathcal{S}$ of cubes $Q\in \mathscr{D}(Q_0)$ such that for a.e. $x\in Q_0$,
\begin{equation}\label{eq:LMOD_Lerner}
  |f(x)-m_f(Q_0)|\le 2\sum_{Q \in \mathcal{S}} \o_{\frac{1}{2^{n+2}}}(f;Q)\chi_{Q}(x).
\end{equation}
\end{theorem}
%

%\section{Pointwise estimate for nonlinear positive dyadic shifts of complexity $m$}
\section{Proof of Theorem \ref{theoremA}}\label{secpointwise}

This section is entirely devoted to the proof of theorem \ref{theoremA}. To that end, some reductions are in order. First, since the operator $\mathcal{T}_{\mathcal{S}}^{k,p_0}$ is (multi)-sublinear, we may assume that $f_i \geq 0$ for $1\leq i \leq m$. Second, by a well known variation of the one-third trick (see, for example, \cite{HLP}), we may exchange centered dilations by dyadic ancestors. More precisely, we may write
$$
\mathcal{T}_{\mathcal{S}}^{k,p_0}\vec{f} (x) \lesssim_{p_0,n} \sum_{j=1}^{c_n} \mathcal{A}^{k,p_0}_{\mathcal{S}_j} \vec{f} (x).
$$
for certain dyadic systems $\mathscr{D}^1, \ldots, \mathscr{D}^{c_n}$, sparse families $\mathcal{S}_j \subset \mathscr{D}^j$ and some dimensional constant $c_n$. Therefore, we may just concentrate on one such operator $\mathcal{A}_{\mathcal{S}}^{k,p_0}$. However, we will consider a slightly more general operator. Namely, given a dyadic system $\mathscr{D}$, we will study operators of the form
$$
\mathcal{A}_{\alpha}^{k,p_0} \vec{f}(x) = \sum_{Q \in \mathscr{D}} \alpha_Q \left[ \prod_{i=1}^m  \langle f_i \rangle_{Q^{(k)},p_0} \right] \chi_Q(x),
$$
where the sequence $\alpha=(\alpha_Q)_Q$ is \textit{Carleson} and normalized:
$$
\sup_{Q\in \mathscr{D}} \sum_{T \in \mathscr{D}(Q)}  \alpha_T |T| =1.
$$
Finally, by the usual density arguments, we may assume that the sequence $\alpha$ is finite, which in particular implies that there exists some cube $P_0 \in \mathscr{D}$ such that $\mathcal{A}_{\alpha}^{k,p_0} = \mathcal{A}_{\alpha,P_0}^{k,p_0}$, that is,
$$
\mathcal{A}_{\alpha}^{k,p_0} \vec{f}(x) = \sum_{Q \in \mathscr{D}, Q^{(k)} \subset P_0} \alpha_Q \left[ \prod_{i=1}^m  \langle f_i \rangle_{Q^{(k)},p_0} \right] \chi_Q(x).
$$
{At the same time, we will assume that each function $f_i$ is supported in the cube $P_0$. The positivity of the operators involved and the density of (say) $L^1_c$ in the quasi-Banach space $\mathbb{X}$ will allow to pass to the limit, as in \cite[chap. 2.3]{CondeThesis}.} The rest of the proof consists of a pointwise estimate of $\mathcal{A}_{\alpha,P_0}^{k,p_0}$ and follows the lines of \cite{CR}. Since it is a bit lengthy, we have chosen to divide it into several steps. We will skip some details in the points where our argument does not differ substantially from that of \cite{CR}.
\\
\vskip 2pt

\noindent \textbf{Step 1. Slicing: reduction to separated scales.} We start the proof separating the scales of $\mathcal{A}^{k,p_0}_{\alpha,P_0}$ as follows:
\begin{align*}
	\mathcal{A}^{k,p_0}_{\alpha,P_0}\vec{f}(x) &=
		\sum_{\ell=0}^{k-1} \sum_{j=1}^\infty \sum_{Q \in \mathscr{D}_{jk+\ell}(P_0)} \alpha_Q \Bigl( \prod_{i=1}^m \langle f_i \rangle_{Q^{(k)},p_0} \Bigr) \chi_Q(x) \\
	&=: \sum_{\ell=0}^{k-1} \mathcal{A}_{\alpha,P_0}^{k,p_0;\ell} \vec{f}(x).
\end{align*}
Now, as in \cite{CR}, we rewrite $\mathcal{A}_{\alpha,P_0}^{k,p_0;\ell}$ as a sum of disjointly supported operators of the form $\mathcal{A}_{\alpha,P}^{k,p_0;0}$. Indeed, we have the expression
\begin{align*}
	\mathcal{A}^{k,p_0}_{\alpha,P_0}\vec{f}(x) &= \sum_{\ell=0}^{k-1} \sum_{P \in \mathscr{D}_\ell(P_0)} \mathcal{A}^{k,p_0;0}_{\alpha,P} \vec{f}(x).
\end{align*}
Therefore, it is enough to prove the following claim: Let $k\geq1$ and $\alpha$ be a normalized Carleson sequence. For nonnegative integrable functions $f_1,\dots, f_m $ on $P_0$, there exists a sparse family $\mathcal{S}$ of cubes in $\mathscr{D}(P_0)$ such that
$$
 \mathcal{A}_{\alpha,P_0}^{k,p_0;0}\vec{f}(x) \leq C \sum_{Q\in \mathcal{S}}\left( \prod_{i=1}^{m}\langle f_i \rangle_{Q,p_0} \right)\chi_{Q}(x),
$$
for some constant $C$ independent of $k$ and the cube $P_0$.
\\
\vskip 2pt

\noindent \textbf{Step 2. Construction of the collection $\mathcal{S}$ for the sliced operator.} We now build the family $\mathcal{S}$. The construction is similar to that in \cite[P. 6]{CR}. We start by defining the quantity
$$
C^* := 2^{2(m+1)} \mathrm{W}(p_0,k),
$$
where
$$
\mathrm{W}{(p_0,k)} = \sup_{\begin{subarray}{c}P \in \mathscr{D}, \alpha \; \mathrm{Carleson} \\ \alpha_Q \not=0 \Rightarrow Q \in \mathscr{D}(P) \end{subarray}} \left\| \mathcal{A}_{\alpha}^{k,p_0;0}\right\|_{L^{p_0}\times L^{p_0} \ldots \times L^{p_0} \to L^{p_0/m,\infty}}.
$$
Also, if $Q\in \mathscr{D}_{kn}(P_0)$ for some $n \geq 0$ define
$$
\gamma_Q = \max_{R \in \mathscr{D}_k(Q)} \alpha_R.
$$
Set also $\Delta_{P_0}=0$. Then, we inductively implement the following selection procedure, starting with the cube $P=P_0$:
\begin{enumerate}
\item If $\Delta_P-\left( \prod_{i=1}^{m}\langle f_i\rangle_{P,p_0} \right)\gamma_{P} <0$, then we choose $P\in \mathcal{S}$ and we set
$$
\Delta_Q = \Delta_P + (C^*-\alpha_Q)\left( \prod_{i=1}^{m}\langle f_i\rangle_{P,p_0} \right).
$$
for all $Q \in \mathscr{D}_k(P)$.
\item If $\Delta_P-\left( \prod_{i=1}^{m}\langle f_i\rangle_{P,p_0} \right)\gamma_{P} \geq0$, then we choose $P\not\in \mathcal{S}$ and we set
$$
\Delta_Q = \Delta_P -\alpha_Q\left( \prod_{i=1}^{m}\langle f_i\rangle_{P,p_0} \right).
$$
\item Go back to (1) for the cubes $Q\in \mathscr{D}_{k}(P)$.
\end{enumerate}
Since the sequence $\alpha$ is finite the procedure ends and yields the family $\mathcal{S}$ that we will use.
\\
\vskip 2pt

\noindent \textbf{Step 3. The family $\mathcal{S}$ is sparse.} To prove sparsity, we will show the following (stronger) claim: fix $P \in \mathcal{S}$, and denote
$$
F(P):=\bigcup_{Q\subsetneq P, Q\in \mathcal{S} } Q.
$$
Then, $|F(P)|\leq \frac{1}{2} |P|$. The claim and its proof are entirely similar to \cite[P.7-8]{CR}. Let $\mathcal{R}$ be the collection of maximal subcubes of $P$ which belong to $\mathcal{S}$. By maximality, for each $x \in R \in \mathcal{R}$ we have
$$
\left(\prod_{i=1}^{m}\langle f_i\rangle_{R,p_0}\right)\gamma_{R} + \mathcal{A}_{\alpha,P}^{k,p_0;0}\vec{f}(x)> C^* \left(\prod_{i=1}^{m}\langle f_i\rangle_{P,p_0}\right).
$$
Now, denote $\mathcal{G}_{P,p_0}\vec{f} = \sum_{R\in \mathcal{R}} \gamma_{R} \left(\prod_{i=1}^{m} \langle f_i\rangle_{R,p_0}\right) \chi_{R}$. Then for all $x\in \mathcal{R}$,
$$
\mathcal{G}_{P,p_0}\vec{f}(x)+\mathcal{A}_{\alpha,P}^{k,p_0;0}\vec{f}(x)> C^* \left(\prod_{i=1}^{m} \langle f_i \rangle_{P,p_0}\right).
$$
Thus we have
\begin{eqnarray*}
|F(P)| & \leq & \left| \left\{ x \in P:  \mathcal{G}_{P,p_0}\vec{f}(x)+\mathcal{A}_{\alpha,P}^{k,p_0;0}\vec{f}(x)> C^* \left(\prod_{i=1}^{m} \langle f_i \rangle_{P,p_0}\right) \right\}  \right| \\
& \leq & \frac{\|\mathcal{G}_{P,p_0}+\mathcal{A}_{\alpha,P}^{k,p_0;0} \|_{L^{p_0}(P)\times \dots \times L^{p_0}(P) \to L^{p_0/m,\infty}(P)}^{p_{0}/m} }{\left( C^* \left(\prod_{i=1}^{m} \langle f_i\rangle_{P,p_0}\right) \right)^{p_{0}/m}}  \left(\prod_{i=1}^{m} \|f_i\|_{L^{p_0}(P)} \right)^{p_{0}/m}\\
& \leq & 2^{mp_0+1} |P|\left(\frac{\|\mathcal{G}_{P,p_0}\|_{L^{p_0}(P)\times \dots \times L^{p_0}(P) \to L^{p_0/m,\infty}(P)}^{p_{0}/m} }{ \left(C^* \right)^{p_{0}/m} } + \frac{\|\mathcal{A}_{\alpha,P}^{k,p_0;0} \|_{L^{p_0}\times \dots \times L^{p_0} \to L^{p_0/m,\infty}(P)}^{p_{0}/m} }{ \left(C^* \right)^{p_{0}/m} }\right)\\
& \leq & \frac{|P|}{2} \left(\frac{\|\mathcal{G}_{P,p_0}\|_{L^{p_0}(P)\times \dots \times L^{p_0}(P) \to L^{p_0/m,\infty}(P)}^{p_{0}/m} }{2} + \frac{1}{2}\right).\\
\end{eqnarray*}
Finally, we observe that the operator $\mathcal{G}_{P,p_0}$ is bounded above by the multi-sublinear operator
$$
\mathcal{P}_{P,p_0} \vec{f} = \sum_{R \in \mathcal{R}} \Bigl(\prod_{i=1}^m \langle f_i \rangle_{R,p_0}\Bigr) \chi_R,
$$
which is contractive from $L^{p_0}(P) \times \ldots \times L^{p_0}(P)$ to $L^{p_0/m,\infty}(P)$. Therefore, the norm of $\mathcal{G}_{P,p_0}$ from $L^{p_0}(P) \times \ldots \times L^{p_0}(P)$ to $L^{p_0/m,\infty}(P)$ is bounded by $1$. This is enough to obtain the assertion.
\\
\vskip 2pt

\noindent \textbf{Step 4. Pointwise bound.} Following the proof of \cite[lemma 2.3]{CR}, one gets the pointwise bound
$$
\mathcal{A}_{\alpha,P_0}^{k,p_0;0}\vec{f}(x) \lesssim_{n,m} \mathrm{W}(p_0,k) \sum_{Q\in \mathcal{S}} \left( \prod_{i=1}^{m} \langle f_i\rangle_{Q,p_0} \right)\chi_{Q}(x).
$$
Therefore, we only need to prove the bound $\mathrm{W}(p_0,k) \lesssim_{p_0,n,m} 1$ and the proof will be complete.
\\
\vskip 2pt

\noindent \textbf{Step 5. Weak type estimate for $\mathcal{A}_{\alpha,P}^{k,p_0}$.} Fix some $P \in \mathscr{D}$ and some normalized Carleson sequence $\alpha$ such that $\alpha_Q \not=0$ only if $Q \in \mathscr{D}(P)$. We need to show that
$$
\|\mathcal{A}_{\alpha,P}^{k,p_0} \|_{L^{p_0} \times \ldots \times L^{p_0} \to L^{p_0/m, \infty}} \lesssim_{n,m,p_0} 1.
$$
To prove it, we first establish an $L^{2p_0m}$ estimate. We will use the estimate \cite{Chen2013}, which reads as follows:
\begin{equation}\label{Complexity.MultilinearCET}
    \Bigl( \sum_{Q \in \mathscr{D}(P)} \alpha_Q \Bigl( \prod_{i=1}^m \frac{1}{|Q|}\int_Q  f_i  \Bigr)^q |Q|  \Bigr)^{\frac1q} \leq \prod_{i=1}^m p_i' \|f_i\|_{L^{p_i}(P)}
\end{equation}
whenever
\begin{equation*}
    \frac{1}{q} = \frac{1}{p_1} + \dots + \frac{1}{p_m}
\end{equation*}
and $\alpha$ is Carleson and normalized. What we will show is
 %(which is indeed the case if $\alpha_Q= 1 \Leftrightarrow Q \in \mathcal{S}$ and $\mathcal{S}$ is sparse).
\[
\|\mathcal{A}_{\alpha,P}^{k,p_0} \vec{f}\|_{L^{2p_0}} \lesssim \prod_{i=1}^m \|f_i\|_{L^{2p_0m}}.
\]
Indeed, we begin by using duality to reduce to showing
	\[
		\int_{P} g(x) \mathcal{A}_{\alpha,P}^{k,p_0} \vec{f}(x) \, dx \lesssim 1
	\]
assuming that $\|f_i\|_{L^{2p_0m}} = \|g\|_{L^{(2p)'}} = 1$ for all $1\leq i \leq m$ and $g\geq 0$. By definition and H\"older's inequality, it is enough to show
	\[
		\Bigl( \sum_{Q \in \mathscr{D}_{\geq m}(P_0)} \alpha_Q \Bigl(\prod_{i=1}^m \langle f_i \rangle_{Q^{(k)},p_0} \Bigr)^{2p_0} |Q| \Bigr)^{1/{2p_0}}
			\Bigl( \sum_{Q \in \mathscr{D}_{\geq k}(P_0)} \alpha_Q \left(\frac{1}{|Q|}\int_Q g \right)^{(2p_0)'} |Q| \Bigr)^{1/(2p_0)'}.
	\]
	The second term can be estimated, using \eqref{Complexity.MultilinearCET} in the linear case, by an absolute constant. For the first term observe that the sequence $\beta_Q$ defined by
	\[
		\beta_Q = \frac{1}{2^{nk}} \sum_{R \in \mathscr{D}_k(Q)} \alpha_R
	\]
	is a Carleson sequence adapted to $P$ of constant 1. Indeed, for any $Q \in \mathscr{D}(P)$, there holds:
	\begin{align*}
		\frac{1}{|Q|} \sum_{R \in \mathscr{D}(Q)} \beta_R|R| &= \frac{1}{|Q|}\sum_{R \in \mathscr{D}(Q)} |R| \frac{1}{2^{nk}} \sum_{T \in \mathscr{D}_k(R)} \alpha_T \\
		&= \frac{1}{|Q|}\sum_{R \in \mathscr{D}(Q)}\sum_{T \in \mathscr{D}_k(R)} \alpha_T |T| \\
		&= \frac{1}{|Q|} \sum_{R \in \mathscr{D}_{\geq k}(Q)} \alpha_R|R| \\
		&\leq 1.
	\end{align*}
Therefore, we can write the first term as
	\[
		\Bigl( \sum_{Q \in \mathscr{D}(P)} \beta_Q \Bigl( \sum_{i=1}^m \langle f_i \rangle_{Q,p_0}  \Bigr)^{2p_0} |Q| \Bigr)^{1/{2p_0}},
	\]
which can also be estimated by \eqref{Complexity.MultilinearCET}, with $p_1=p_2=\ldots = p_m=2m$, $q=2$:	
\begin{eqnarray*}
\Bigl( \sum_{Q \in \mathscr{D}(P)} \beta_Q \Bigl( \prod_{i=1}^m \langle f_i \rangle_{Q,p_0}  \Bigr)^{2p_0} \Bigr)^{1/{2p_0}} & = & \Bigl( \sum_{Q \in \mathscr{D}(P)} \beta_Q \Bigl( \prod_{i=1}^m \langle |f_i|^{p_0} \rangle_{Q,1}  \Bigr)^{2} |Q| \Bigr)^{1/{2p_0}} \\
& \lesssim_{p_0,m} & \prod_{i=1}^m \||f_i|^{p_0}\|_{2m}^{1/p_0} \lesssim 1.\\
\end{eqnarray*}
Combining both terms we arrive at the strong type result that we want.
\vskip 2pt
Now we can prove our weak type estimate. That is, we want to show that
\begin{equation*}
	\sup_{\lambda > 0} \lambda |\{x: \mathcal{A}_{\alpha,P}^{k,p_0} \vec{f}(x) > \lambda^{m/p_0} \}| ^{m/p_0} \lesssim_{n,m,p_0} \prod_{i=1}^m \|f_i\|_{L^{p_0}}.
\end{equation*}
By homogeneity we can assume $\|f_i\|_{L^{p_0}} = 1$ and $f_i\geq 0$ for $1 \leq i \leq m$. We will use the previous strong bound and a standard Calder\'on-Zygmund decomposition of the positive tuple $(f_1^{p_0}, \ldots, f_m^{p_0})$ that we explain now.

We need the following version of the dyadic maximal operator
$$
\mathcal{M}^\mathscr{D}_{p_0} g(x) = \sup_{x \in Q \in \mathscr{D}}  \langle g \rangle_{Q,p_0}.
$$
For $1\leq i \leq m$, denote
$$
\Omega_i = \{x\in P: \mathcal{M}^\mathscr{D}_{p_0} f_i(x) > \lambda^{1/m}\}.
$$
If $\langle f_i \rangle_{P,p_0} > \lambda^{1/m}$ then we have
	\[
		|P| \lambda^{p_0/m} < \|f_i\|_{L^p_0}^{p_0},
	\]
and the estimate follows by the homogeneity assumption. Therefore, we can assume $\langle f_i \rangle_{P,p_0} \leq \lambda^{1/m}$ for all $1\leq i\leq m$. But then, we can write $\Omega_i$ as a union of cubes in a collection $\mathcal{R}_i$ consisting of pairwise disjoint dyadic (strict) subcubes $R$ of $P$ with the property
	\[
		\langle f_i \rangle_{R,p_0} > \lambda^{1/m} \quad \text{and} \quad \langle f_i \rangle_{R^{(1)},p_0} \leq \lambda^{1/m}, \; R \in \mathcal{R}_i.
	\]
For each $1\leq i \leq m$ let $b_i = \sum_{R \in \mathcal{R}_i} b_i^R$, where	
\[
		b_i^R(x) := \bigl( f_i^{p_0}(x) - \langle f_i \rangle_{R,p_0}^{p_0} \bigr) \chi_R(x).
	\]
	We now let $g_i = f_i^{p_0}-b_i$. Observe that we have
	\[
		|g_i(x)| \lesssim \lambda^{p_0/m}, \; \|g_i\|_{L^1} \lesssim \|f_i\|_{L^{p_0}}=1
	\]
	as well as
	\[
		|\Omega_i| = \sum_{R \in \mathcal{R}_i} |R| \leq \frac{1}{ \lambda^{p_0/m}}.
	\]
	
Set $\Omega = \cup_i \Omega_i$. Now we have
	\begin{align}
		|\{x: \, \mathcal{A}_{\alpha,P}^{k,p_0} \vec{f}(x) > \lambda\}| &\leq |\Omega| +
			|\{x\in \mathbb{R}^n \setminus \Omega: \, \mathcal{A}_{\alpha,P}^{k,p_0} \vec{f}(x) > \lambda\}| \notag \\
		&\leq  \frac{m}{ \lambda^{p_0/m}} + |\{x\in \mathbb{R}^n\setminus \Omega: \, \mathcal{A}_{\alpha,P}^{k,p_0} \vec{f}(x) > \lambda\}|. \label{Complexity.LebesgueWeakType.eq1}
	\end{align}
To estimate the second term above observe that we have
$$
\langle f_i \rangle_{Q,p_0}^{p_0} \leq \left| \left(\frac{1}{|Q|} \int_Q g_i \right) \right|+ \left| \left(\frac{1}{|Q|} \int_Q b_i \right)\right|.
$$
Therefore, by the concavity of the function $x \mapsto |x|^{\frac{1}{p_0}}$ we obtain
$$
\mathcal{A}_{\alpha,P}^{k,p_0} \vec{f}(x) \leq  |\mathcal{A}|_{\alpha,P}^{k,p_0} \vec{g}(x) + \sum_{j=1}^{2^m-1} |\mathcal{A}|_{\alpha,P}^{k,p_0} (h_1^j, \ldots, h_m^j)(x),
$$
where we have denoted $\vec{g} = (g_i)_{1\leq i \leq k}$, $h_i^j$ is either $g_i$ or $b_i$ and for each $j$, there is at least one $1\leq i\leq m$ such that $h_i^j=b_i$. Also, we have used the notation
$$
|\mathcal{A}|_{\alpha,P}^{k,p_0} \vec{h}(x) = \sum_{Q \in \mathscr{D}(P), Q^{k} \subset P} \alpha_Q  \prod_{i=1}^m \left| \left(\frac{1}{|Q^{(k)}|} \int_{Q^{(k)}} h_i \right) \right|^{\frac{1}{p_0}} \chi_Q(x).
$$
If, $h_i^j=b_i$, then for all $x \notin \Omega_i$ we can see that $|\mathcal{A}|_{\alpha,P}^{k,p_0} (h_1^j, \ldots, h_m^j)(x)=0$ because of the average $0$ of each $b_i^R$. With this fact we can see that the second term in \eqref{Complexity.LebesgueWeakType.eq1} is actually identical to
	\[
		|\{x\in \mathbb{R}^n\setminus \Omega: \, |\mathcal{A}|_{\alpha,P}^{k,p_0} \vec{g}(x) > \lambda\}|.
	\]	

But now we can use the $L^{2p_0}$ bound. Denoting $|\vec{g}|^{1/p_0} = (|g_1|^{1/p_0}, \dots , |g_k|^{1/p_0})$, we have
	\begin{align*}
		|\{x\in \mathbb{R}^n \setminus \Omega: \, |\mathcal{A}|_{\alpha,P}^{k,p_0} \vec{g}(x) > \lambda\}| &\leq \frac{1}{\lambda^{2p_0}} \||\mathcal{A}|_{\alpha,P}^{k,p_0} \vec{g}\|_{L^{2p_0}}^{2p_0} \\
		& \leq \frac{1}{\lambda^{2p_0}} \|\mathcal{A}_{\alpha,P}^{k,p_0} |\vec{g}|^{1/p_0}\|_{L^{2p_0}}^{2p_0} \\
		&\lesssim \frac{1}{\lambda^{2p_0}} \prod_{i=1}^m  \||g_i|^{1/p_0}\|_{L^{2p_0m}}^{2p_0} \\
		& \lesssim \frac{1}{\lambda^{p/m}} \prod_{i=1}^m \|g_i\|_{L^{1}}^{1/m}\\
		& \lesssim \frac{1}{\lambda^{p/m}}.
	\end{align*}
	
Putting both estimates together we arrive at the desired result. This completes the proof of theorem \ref{theoremA}.

\section{Proof of Theorem \ref{theoremB} and Theorem \ref{theoremC}}
\label{secapp}

\begin{proof}[Proof of Theorem~\ref{theoremB}]
For this, we borrow some ideas from \cite[Theorem 3.2]{LMS}, where the case $p_0=1$ is considered. Throughout the proof, we set $a=p/p_0$ and $a_i=p_i/p_0$ for $i=1,\dots,m$. Let $\sigma_i=w_i^{1-a_i'}$, $\vec{f}_{\sigma,p_0} =(f_1\sigma_1^{1/p_0},\dotsc, f_m\sigma_m^{1/p_0})$ and $f_i\geq 0$. We have $\sigma_i, \nw \in A_\infty$ (see \cite[Theorem 3.6]{LOPTT}). It suffices to prove that
\begin{equation}\label{eq1-thm1.1}
\|\mathcal{A}^{p_0}_{\mathscr{D},S}(\vec{f}_{\sigma,p_0})\|_{L^p(\nu_{\vec{w}})} \lesssim [\vec{w}]_{A_{\vec{P}/p_0}}^{\max(1,\tfrac{a_1'}{p_0a},\dotsc,\tfrac{a_m'}{p_0a})} \prod_{i=1}^m \|f_i\|_{L^{p_i}(\sigma_i)}.
\end{equation}

By definition, for any cube $Q\subset \mathbb{R}^n$, we have
\begin{equation}\label{eq:multiap_LOPTT}
[\vec{w}]_{A_{\vec{P}/p_0}}\geq \Big(\frac{1}{|Q|}\int_Q\nu_{\vec w}\Big)\prod_{j=1}^m\Big(\frac{1}{|Q|}\int_Q w_j^{1-a'_j}\Big)^{a/a'_j}.
\end{equation}
Denote $\beta=\max(1,\tfrac{a_1'}{p_0a},\dotsc,\tfrac{a_m'}{p_0a})$, and assume that $0\leq g\in L^{p'}(\nu_{\vec{w}})$. We have
$$
\int_{\mathbb{R}^n}\mathcal{A}^{p_0}_{\mathscr{D},S}(\vec{f}_{\sigma,p_0}) g\nu_{\vec{w}}=\sum_{Q\in \mathcal{S}}\int_Q g\nw \times \left(\prod_{i=1}^m \Big(\f{1}{|Q|}\int_Q|f_i|^{p_0}\sigma_i\Big)^{1/p_0}\right)
$$
From this and the definition of $[\vec{w}]_{A_{\vec{P}/p_0}}$, we obtain
$$
\begin{aligned}
\sum_{Q\in \mathcal{S}}\int_Q g\nw& \times \Big(\prod_{i=1}^m\f{1}{|Q|}\int_Q |f_i|^{p_0}\sigma_i\Big)^{1/p_0}\\
&\leq [\vec{w}]_{A_{\vec{P}}}^{\beta}\sum_{Q\in \mathcal{S}}\f{|Q|^{m(\beta a-1/p_0)}}{\nw(Q)^{\beta-1}\prod_{i=1}^m \sigma_i(Q)^{(\beta a/a'_i-1/p_0)}}\times \Big(\f{1}{\nw(Q)}\int_Q g\nw\Big)\\
& \ \ \ ~ ~ ~   \times \Big(\prod_{i=1}^m \f{1}{\sigma_i(Q)}\int_Q |f_{i}|^{p_0}\sigma_i\Big)^{1/p_0}\\
&\leq 2^{m(\beta a-1/p_0)}[\vec{w}]_{A_{\vec{P}}}^{\beta}\sum_{Q\in \mathcal{S}}\f{|E_Q|^{m(\beta a-1/p_0)}}{\nw(Q)^{\beta-1}\prod_{i=1}^m \sigma_i(Q)^{(\beta a/a'_i-1/p_0)}}\times \Big(\f{1}{\nw(Q)}\int_Q g\nw\Big)\\
& \ \ \ ~ ~ ~  \times \Big(\prod_{i=1}^m \f{1}{\sigma_i(Q)}\int_Q |f_{i}|^{p_0}\sigma_i\Big)^{1/p_0}\\
%&\leq 2^{mq(\beta p-1)}[\vec{w}]_{A_{\vec{P}}}^{\beta q}\sum_{Q\in \mathcal{S}}\f{|E_Q|^{mq(\beta p-1)}}{\nw(Q)^{\beta q-1}\prod_{i=1}^m \sigma_i(Q)^{q(\beta %p/p'_i-1)}}\times M_{\nw}^{\mathscr{D}}(g) \times \prod_{i=1}^m [M_{\sigma_i}^{\mathscr{D}}(f_i)]^q\\
&\leq 2^{m(\beta a-1/p_0)}[\vec{w}]_{A_{\vec{P}}}^{\beta}\sum_{Q\in \mathcal{S}}\f{|E_Q|^{m(\beta a-1/p_0)}}{\nw(E_Q)^{\beta-1}\prod_{i=1}^m \sigma_i(E_Q)^{(\beta a/a'_i-1/p_0)}}\times \Big(\f{1}{\nw(Q)}\int_Q g\nw\Big)\\
& \ \ \ ~ ~ ~  \times \Big(\prod_{i=1}^m \f{1}{\sigma_i(Q)}\int_Q |f_{i}|^{p_0}\sigma_i\Big)^{1/p_0},\\
\end{aligned}
$$
where in the last inequality we used the facts $ \nw(Q) \geq \nw(E_Q)$, $\sigma_i(Q)\geq \sigma_i(E_Q)$ and  the positivity of the exponents. On the other hand, by H\"older's inequality, we have
\begin{equation}\label{eq2-thm1.1}
|E_Q|=\int_{E_Q}\nw^{\f{1}{ma}}\prod_{i=1}^m\sigma_i^{\f{1}{ma_i'}}\leq \nw(E_Q)^{\f{1}{ma}}\prod_{i=1}^m\sigma_i(E_Q)^{\f{1}{ma_i'}}.
\end{equation}
Inserting this into the estimate above we conclude that
$$
\begin{aligned}
\sum_{Q\in \mathcal{S}}&\int_Q g\nw \times \Big(\prod_{i=1}^m\f{1}{|Q|}\int_Q |f_i|^{p_0}\sigma_i\Big)^{1/p_0}\\
&\leq 2^{m(\beta a-1/p_0)}[\vec{w}]_{A_{\vec{P}}}^{\beta}\sum_{Q\in \mathcal{S}}\f{\nw(E_Q)^{(\beta a -1/p_0)/a}\prod_{i=1}^m\sigma_i(E_Q)^{(\beta a -1/p_0)/a_i'}}{\nw(E_Q)^{\beta-1}\prod_{i=1}^m \sigma_i(E_Q)^{(\beta a/a'_i-1/p_0)}}\times \Big(\f{1}{\nw(Q)}\int_Q g\nw\Big)\\
& \ \ \ ~ ~ ~  \times \Big(\prod_{i=1}^m \f{1}{\sigma_i(Q)}\int_Q |f_{i}|^{p_0}\sigma_i\Big)^{1/p_0}\\
&\leq 2^{m(\beta a-1/p_0)}[\vec{w}]_{A_{\vec{P}}}^{\beta}\sum_{Q\in \mathcal{S}}\nw(E_Q)^{1-\f{1}{ap_0}}\prod_{i=1}^m\sigma_i(E_Q)^{\f{1}{p_0a_i}}\times \Big(\f{1}{\nw(Q)}\int_Q g\nw\Big)\\
& \ \ \ ~ ~ ~  \times \Big(\prod_{i=1}^m \f{1}{\sigma_i(Q)}\int_Q |f_{i}|^{p_0}\sigma_i\Big)^{1/p_0}\\
&= 2^{mq(\beta a-1/p_0)}[\vec{w}]_{A_{\vec{P}}}^{\beta}\sum_{Q\in \mathcal{S}}\nw(E_Q)^{\f{1}{p'}}\prod_{i=1}^m\sigma_i(E_Q)^{\f{1}{p_i}}\Big(\f{1}{\nw(Q)}\int_Q g\nw\Big)\\
&\quad \quad \quad  \times \Big(\prod_{i=1}^m \f{1}{\sigma_i(Q)}\int_Q |f_i|^{p_0}\sigma_i\Big)^{1/p_0}\\
&= 2^{mq(\beta a-1/p_0)}[\vec{w}]_{A_{\vec{P}}}^{\beta }\sum_{Q\in \mathcal{S}} \left[ \Big(\f{1}{\nw(Q)}\int_Q g\nw\Big)\nw(E_Q)^{\f{1}{p'}} \right]\\     &\quad \quad \quad \times \left[\prod_{i=1}^m \Big(\f{1}{\sigma_i(Q)}\int_Q |f_i|^{p_0}\sigma_i\Big) \sigma_i(E_Q)^{\f{p_0}{p_i}}          \right]^{1/p_0}.
\end{aligned}
$$
This, together with H\"older's inequality and the disjointness of the family $\{E_Q\}_{Q\in \mathcal{S}}$ yields
$$
\begin{aligned}
\sum_{Q\in \mathcal{S}}\int_Q g\nw \times \Big(\prod_{i=1}^m\f{1}{|Q|}\int_Q& |f_i|^{p_0}\sigma_i\Big)^q\\
&\leq 2^{mq(\beta a-1/p_0)}[\vec{w}]_{A_{\vec{P}}}^{\beta} \Big[\sum_{Q\in \mathcal{S}} \Big(\f{1}{\nw(Q)}\int_Q g\nw\Big)^{p'}\nw(E_Q)\Big]^{\f{1}{p'}}\\
 & \ \ \ \times \prod_{i=1}^m \Big[\sum_{Q\in \mathcal{S}}\Big(\f{1}{\sigma_i(Q)}\int_Q |f_i|^{p_0}_{i}\sigma_i\Big)^{\f{p_i}{p_0}}\sigma_i(E_Q)\Big]^{\f{1}{p_i}}\\
&\leq 2^{m(\beta a-1/p_0)}[\vec{w}]_{A_{\vec{P}}}^{\beta}\|M_{\nw}^{\mathscr{D}}(g)\|_{L^{p'}(\nw)} \times \prod_{i=1}^m \|M_{\sigma_i}^{\mathscr{D}}(|f_i|^{p_0})\|^{1/p_0}_{L^{p_i/p_0}(\sigma_i)}\\
&{\lesssim} 2^{m(\beta a-1/p_0)}[\vec{w}]_{A_{\vec{P}}}^{\beta} \|g\|_{L^{p'}(\nw)} \times \prod_{i=1}^m\|f_i^{p_0}\|^{1/p_0}_{L^{p_i/p_0}(\sigma_i)}\\
&{=} 2^{m(\beta a-1/p_0)}[\vec{w}]_{A_{\vec{P}}}^{\beta} \|g\|_{L^{p'}(\nw)} \times \prod_{i=1}^m\|f_i\|_{L^{p_i}(\sigma_i)},
\end{aligned}
$$
applying (\ref{bbs}) to get last inequality. This proves \eqref{eq1-thm1.1}.
\end{proof}

The following proposition plays an important role in proving Theorem \ref{theoremC}.
\begin{prop}\label{prop4.2} Let $T$ satisfy (H1) and (H2). Then, for any cube $Q\subset\Rn$, we have

    $$
\omega_\lambda(T\vec{f}, Q)\leq c(T,\lambda,m, n)\sum_{\ell=0}^\vc 2^{-\ell\delta_0}\Big(\prod_{i=1}^m\f{1}{|2^\ell Q|}\int_{2^\ell Q}|f_i(y)|^{p_0}dy\Big)^{1/p_0},
$$
where $\delta_0=\delta-n/p_0$.
\end{prop}
\begin{proof}
The proof of this proposition is standard. For reasons of completeness, we sketch it here. For each $i=1,\dots,m$, we define $f_i^0=f_i\chi_{Q^*}$ and $f_i^\vc=f_i -f_i^0$. Setting Let $\vec{f}^0=(f_1\chi_{4Q},\ldots, f_m\chi_{4Q})$,
then we have
    \begin{equation}\label{eq1-prop4.2}
    T(\vec f)(z) = T(\vec {f}^0)(z)+\sum_{\vec{\alpha}\in \mathcal{I}_0}  T(f_1^{\a_1},\ldots,f_m^{\a_m})(z),
  \end{equation}
where $\mathcal{I}_0:=\{\vec{\alpha}=(\alpha_1,\dots,\alpha_m): \, \alpha_i\in\{0,\vc\}, \ \text{ and at least one $\alpha_i\neq 0$}\}$. We first observe that
$$
\begin{aligned}
\Big[&\Big(T(\vec{f})-\sum_{\vec{\alpha}\in \mathcal{I}_0}  T(f_1^{\a_1},\ldots,f_m^{\a_m})(x_0)\Big)\chi_Q\Big]^*(\lambda|Q|)\\
&\leq 2(T(\vec{f^0})\chi_Q)^*(\lambda|Q|/2)+2\Big\|\sum_{\vec{\alpha}\in \mathcal{I}_0} T(f_1^{\a_1},\ldots,f_m^{\a_m})(\cdot)-\sum_{\vec{\alpha}\in \mathcal{I}_0}  T(f_1^{\a_1},\ldots,f_m^{\a_m})(x_0)\Big\|_{L^\vc(Q)},
\end{aligned}
$$
where $x_0$ is the center of $Q$. Since $T$ maps $L^{p_0}\times\ldots\times L^{p_0}$ into $L^{p_0/m,\vc}$, we have
$$
\begin{aligned}
(T\vec{f}^0)^*(\lambda|Q|)&\leq C_{n,T, \lambda}\|T\vec{f}^0\|_{L^{p_0/m,\vc}(Q,\f{dx}{|Q|})}\\
&\leq C_{n,T, \lambda}\Big(\prod_{i=1}^m\f{1}{|4Q|}\int_{4Q}|f_i(y)|^{p_0}dy\Big)^{1/p_0}.
\end{aligned}
$$

On the other hand, for $x\in Q$, the argument in \cite[Theorem 3.1]{BD} proved that
\begin{equation}\label{eq1-prop}
\begin{aligned}
\Big|\sum_{\vec{\alpha}\in \mathcal{I}_0} T(f_1^{\a_1},\ldots,f_m^{\a_m})(x)&-\sum_{\vec{\alpha}\in \mathcal{I}_0}  T(f_1^{\a_1},\ldots,f_m^{\a_m})(x_0)\Big|\\
&\leq C_{n,m, T}\sum_{\ell=0}^\vc 2^{-\ell\delta_0}\Big(\prod_{i=1}^m\f{1}{|2^\ell Q|}\int_{2^\ell Q}|f_i(y)|^{p_0}dy\Big)^{1/p_0}
\end{aligned}
\end{equation}
with $\delta_0=\delta-n/p_0$. Taking the last two estimates into account we obtain
$$
\begin{aligned}
\Big[\Big(T(\vec{f})&-\sum_{\vec{\alpha}\in \mathcal{I}_0}  T(f_1^{\a_1},\ldots,f_m^{\a_m})(x_0)\Big)\chi_Q\Big]^*(\lambda|Q|)\\
&\leq  C_{n,m,T,\lambda}\sum_{\ell=0}^\vc 2^{-\ell\delta_0}\Big(\prod_{i=1}^m\f{1}{|2^\ell Q|}\int_{2^\ell Q}|f_i(y)|^{p_0}dy\Big)^{1/p_0}.
\end{aligned}
$$
This completes our proof.
\end{proof}

At this stage, Theorem \ref{theoremC} follows immediately from Theorem \ref{theoremB} via the following short argument:

\begin{theorem}
 \label{th2} Let $T$ satisfy (H1) and (H2) and let $\mathbb{X}$ be a quasi-Banach function space. Then we have
   \begin{equation}\label{eq:22}
    \|T(\vec f\,)\|_{\mathbb{X}}\lesssim_{T,m,n}\sup_{{\mathscr{D}},{\mathcal S}}\|   \mathcal{A}^{p_0}_{\mathscr{D},S}(\vec{f}) \|_{\mathbb{X}}.
  \end{equation}
  \end{theorem}
\begin{proof}
From Theorem \ref{decom1} and Proposition \ref{prop4.2}, for $Q_0\in \mathscr{D}$, we can pick a sparse family $\mathcal{S}(Q_0)\in \mathscr{D}(Q_0)$ so that
$$
\begin{aligned}
|T\vec{f}(x)&-m_{T\vec{f}}(Q_0)|\leq c(T,n,m)\sum_{Q\in \mathcal{S}(Q_0)}\sum_{\ell=0}^\vc 2^{-\ell\delta_0}\Big(\prod_{i=1}^m\f{1}{|2^\ell Q|}\int_{2^\ell Q}|f_i(y)|^{p_0}dy\Big)^{1/p_0}\chi_{Q}(x),
\end{aligned}
$$
for a.e. $x\in Q_0$. Since $T$ maps $L^{p_0}\times\ldots\times L^{p_0}$ into $L^{p_0/m,\vc}$, $\lim_{|Q_0|\to \vc}m_{T\vec{f}}(Q_0)=0$ provided $f_i\in L^{p_0}, i=1,\dots,m$. Hence, using Fatou's lemma, we obtain that
$$
\|T\vec{f}\|_{\mathbb{X}}\lesssim_{T,n,m}\sum_{\ell=0}^\vc 2^{-\ell\delta_0} \sup_{\mathcal{S}\in \mathscr{D}}\|\mathcal{T}^{p_0}_{\mathcal{S},\ell}\vec{f}\|_{\mathbb{X}}.
$$
This along with Theorem \ref{theoremA} implies that
$$
\|T(\vec f\,)\|_{\mathbb{X}}\lesssim_{T,m,n}\sup_{{\mathscr{D}},{\mathcal S}}\|   \mathcal{A}^{p_0}_{\mathscr{D},S}(\vec{f}) \|_{\mathbb{X}}.
$$
\end{proof}

\section{Applications to certain singular integral operators with nonsmooth kernels}\label{App-sec}
\subsection{Linear Fourier Multipliers}
Let $m$ be a bounded function on $\mathbb{R}^n$. We define the multiplier operator $T_m$ by setting
$$
(T_mf) {\hat{}}(x)=m(x)\hat{f}(x)
$$
where $\hat{f}$ is the Fourier transform of $f$.

\medskip

Let $s\geq 1$, $l$ be a positive integer and $\alpha=(\alpha_1,\ldots,\alpha_n)$ be a multi-index of nonnegative integers $\alpha_j$ with the length $|\alpha|=|\alpha_1|+\ldots+|\alpha_n|$. According to \cite{KW}, we say that the function $m\in M(s,l)$ if
\begin{equation}\label{cond on m}
\sup_{R>0}\Big(R^{s|\alpha|-n}\int_{R<|x|<2R}|\partial^\alpha m(x)|^s dx\Big)^{1/s}< +\vc \ \text{for all $|\alpha|\leq l$}.
\end{equation}

In \cite{Ho}, H\"ormander showed that if $m\in M(2,l)$ and $l>n/2$, then the associated operators $T_m$ are bounded on $L^p$ for $1<p<\vc$. The condition $m\in M(s,l)$ with $s\geq 2$ and $l>n/s$ was considered by \cite{CT}. In \cite{KW}, the authors consider the class of $m\in M(s,l)$  with  $s\leq 2$ and $l>n/s$. The following result is a direct consequence of Theorem 1 in \cite{KW}.

\vskip0.5cm

\begin{theorem}\label{thmTm}
Let $1<s\leq 2$ and $m\in M(s,l)$  with $n/s<l<n$. Then $T_m$ is bounded on $L^p$ for $1<p<\infty$.
\end{theorem}

Moreover, the folloeing estimate follows from  \cite[Lemma 1]{KW}:

\begin{lemma}\label{kernelestimates 2}
Let $1<s\leq 2$ and $m\in M(s,l)$  with $n/s<l<n$. Then for any $p_0>n/l$ there exists $\epsilon>0$ such that for any ball $B$, and $x, \overline{x}\in B$, there holds
$$
\begin{aligned}
\Big(\int_{S_k(B)}|K(x,y)-K(\overline{x},y)|^{p'_0}dy\Big)^{1/p'_0}\leq C \f{2^{-k\epsilon}}{(2^kr_B)^{n/p_0}}, \; p_0 < \frac{n}{l},
\end{aligned}
$$
for all $k\geq 2$, where $K(x,y)$ is the associated kernel of $T_m$.
\end{lemma}
Therefore, $T_m$ satisfies conditions (H1) and (H2). Then, as a consequence of Theorem \ref{theoremC}, we obtain the following result.
\begin{theorem}\label{thm1-App}
Let $1<s\leq 2$ and $m\in M(s,l)$  with $n/s<l<n$. For any $n/l<p_0<\vc$, the following hold true:
\begin{enumerate}[(a)]
\item For $p_0<p<\vc$ and $w\in A_{p/p_0}$, we have
$$
\|T_mf\|_{L^p(w)}\leq C_{T_m,p,p_0}[w]^{\max\{1,\f{1}{p-p_0}\}}_{A_{p/p_0}}\|f\|_{L^p(w)}.
$$

\item For $1<p<p_0'$ and $w\in A_{p}\cap RH_{(p'_0/p)'}$, we have
$$
\|T_mf\|_{L^p(w)}\leq C_{T_m,p,p_0}[w]_{RH_{(p'_0/p)'}}^{\max\{p'-1,\f{p'-1}{p'-p_0}\}}  [w]_{A_p}^{\max\{p'-1,\f{p'-1}{p'-p_0}\}}\|f\|_{L^p(w)}.
$$
\end{enumerate}
\end{theorem}
\begin{proof}
\noindent (a) From Theorem \ref{thmTm} and Lemma \ref{kernelestimates 2}, $T_m$ satisfies (H1) and (H2) for $p_0$. As a consequence of Theorem \ref{theoremC} we get that
$$
\|T_mf\|_{L^p(w)}\leq C_{T_m,p,p_0}[w]^{\max\{1,\f{(p/p_0)'}{p}\}}_{A_{p/p_0}}\|f\|_{L^p(w)}=C_{T_m,p,p_0}[w]^{\max\{1,\f{1}{p-p_0}\}}_{A_{p/p_0}}\|f\|_{L^p(w)}.
$$

\noindent (b) For $w\in A_{p}\cap RH_{(p'_0/p)'}$, we now claim that
\begin{equation}\label{eq1-application}
[\sigma]_{A_{p'/p_0}}\leq  [w]_{RH_{(p'_0/p)'}}^{p'-1}[w]_{A_p}^{p'-1}:=[w]_{RH_{(p'_0/p)'}}^{\f{1}{p-1}}[w]_{A_p}^{\f{1}{p-1}},
\end{equation}
where $\sigma =w^{1-p'}$.

Once \eqref{eq1-application} is proved, the statement (b) follows immediately by duality.

To prove \eqref{eq1-application}, we write
$$
[\sigma]_{A_{p'/p_0}}=\sup_{Q}\Big(\fint_Q \sigma\Big)\Big(\fint_Q \sigma^{1-(p'/p_0)'}\Big)^{p'/p_0-1}.
$$
This along with the fact that
$$
(1-p')\Big[1-\Big(\f{p'}{p_0}\Big)'\Big]=\f{p_0}{p-p_0(p-1)}=\Big(\f{p'_0}{p}\Big)'
$$
implies that
$$
[\sigma]_{A_{p'/p_0}}=\sup_{Q}\Big(\fint_Q w^{1-p'}\Big)\Big(\fint_Q w^{(p_0'/p)'}\Big)^{p'/p_0-1}.
$$
Using the facts that $w\in RH_{(p_0'/p)'}$ and $\displaystyle \Big[\f{p'}{p_0}-1\Big]\Big(\f{p_0'}{p}\Big)'=\f{1}{p-1}$, we obtain that
$$
\begin{aligned}
[\sigma]_{A_{p'/p_0}}&\leq [w]_{RH_{(p_0'/p)'}}^{\f{1}{p-1}} \sup_{Q}\Big(\fint_Q w^{1-p'}\Big)\Big(\fint_Q w\Big)^{\f{1}{p-1}}\\
&\leq [w]_{RH_{(p_0'/p)'}}^{\f{1}{p-1}}[w]_{A_p}^{\f{1}{p-1}}.
\end{aligned}
$$
This proves \eqref{eq1-application}.
\end{proof}

\begin{remark}
Note that it was proved in \cite[Theorem 1]{KW} that under the same assumptions as Theorem \ref{thm1-App}, $T_m$ is bounded on $L^p(w)$ for $n/l<p<\vc$ and $w\in A_{pl/n}$ and hence by duality $T_m$ is bounded on $L^p(w)$ for $1<p<(n/l)'$ and $w\in A_{p}\cap RH_{((n/l)'/p)'}$. Hence, it is reasonable to expect that the weighted bounds in Theorem \ref{thm1-App} still hold for $p_0=n/l$. It is not clear whether or not this conjecture is true and we leave it as an open problem.
\end{remark}
\subsection{Riesz transforms related to Schr\"odinger operators}
Let $L=-\Delta +V$ be  the Schr\"odinger operators on $\mathbb{R}^n$ with $n\geq 3$ where the potential $V$ is in the reverse H\"older class $RH_{q}$ for some $q>n/2$. Note that in the case $V\in RH_q, q\geq n$ the Riesz transforms $\nabla L^{-1/2}$ and $L^{-1/2}\nabla$ turn out to be Calder\'on-Zygmund operators. See for examples \cite{Shen}. This is a reason why we restrict ourself to consider the case $V\in RH_q$ with $n/2<q<n$. We now recall the following result concerning the boundedness of the Riesz transforms $\nabla L^{-1/2}$ and $L^{-1/2}\nabla$ in \cite{Shen}.
\begin{theorem}\label{thmRiesz}
Let $L=-\Delta +V$ be  the Schr\"odinger operators on $\mathbb{R}^n$ with $n\geq 3$. Assume that $V\in RH_q, n/2<q<n$. Let $p_0=\f{qn}{n-q}$. Then we have
\begin{enumerate}[(a)]
\item $L^{-1/2}\nabla$ is bounded on $L^p$ for $p'_0\leq p<\vc$.
\item $\nabla L^{-1/2}$ is bounded on $L^p$ for $1<p\leq p_0$.
\end{enumerate}
\end{theorem}
We now apply Theorem \ref{theoremC} to get the weighted bounds for these operators.
\begin{theorem}
Let $L=-\Delta +V$ be  the Schr\"odinger operators on $\mathbb{R}^n$ with $n\geq 3$. Assume that $V\in RH_q, n/2<q<n$. Let $p_0=\f{qn}{n-q}$. Then we have
\begin{enumerate}[(a)]
\item For $p'_0<p<\vc$ and $w\in A_{p/p'_0}$, we have
$$
\|L^{-1/2}\nabla f\|_{L^p(w)}\leq C_{L,p,q}[w]^{\max\{1,\f{1}{p-p_0'}\}}_{A_{p/p'_0}}\|f\|_{L^p(w)}.
$$

\item For $1<p<p_0$ and $w\in A_{p}\cap RH_{(p_0/p)'}$, we have
$$
\|\nabla L^{-1/2}f\|_{L^p(w)}\leq C_{L,p,q}[w]_{RH_{(p_0/p)'}}^{\max\{p'-1,\f{p'-1}{p'-p_0'}\}}  [w]_{A_p}^{\max\{p'-1,\f{p'-1}{p'-p_0'}\}}\|f\|_{L^p(w)}.
$$
\end{enumerate}
\end{theorem}
\begin{proof}
(a) Let $K(x,y)$ be an associated kernel of the Riesz transform $L^{-1/2}\nabla$, according to the proof of Theorem 1.6 (iii) in \cite{GLP}
 there exists $\epsilon>0$ such that for any ball $B$, and $x, \overline{x}\in B$, there holds
$$
\begin{aligned}
\Big(\int_{S_k(B)}|K(x,y)-K(\overline{x},y)|^{p_0}dy\Big)^{1/p_0}\leq C \f{2^{-k\epsilon}}{(2^kr_B)^{n/p'_0}},
\end{aligned}
$$
for all $k\geq 2$. Hence, the statement (a) follows immediately from Theorem \ref{theoremC}.

\noindent (b) Part (b) follows from (a) and duality argument used in Theorem \ref{thm1-App}.
\end{proof}

\begin{remark}
It is worth noting that our approach is still applicable to obtain the weighted bounds for other Riesz transforms such that $V^{1/2}L^{-1/2}, L^{-1/2}V^{1/2}$, $VL^{-1}$ and $L^{-1}V$.
\end{remark}
\subsection{Multilinear Fourier multiplier}

Another application of Theorem \ref{theoremC} is to obtain the weighted bounds for multilinear Fourier multiplier operators. %Before coming to the details, we consider the linear case first. Let $m\in L^{\vc}(\mathbb{R}^n)$. The Fourier multiplier operator $T_m$ is defined by
%$$
%T_mf(x)=\f{1}{(2\pi)^n}\int_{\mathbb{R}^n}e^{ix\cdot \xi}m(\xi)\hat{f}(x)d\xi
%$$
%for all Schwart functions $f\in S(\mathbb{R}^n)$, where $\hat{f}$ is the Fourier transform of $f$.

%It is well-known that if $m$ satisfies the following condition
%$$
%|\p_\xi^\alpha m(\xi)|\leq C_\alpha|\xi|^{-\alpha} \ \ \text{for all $\alpha\leq [n/2]+1$}
%$$
%then $T_m$ is bounded on $L^p$ for all $1<p<\vc$ see e. g. \cite[Corollary 8.11]{Dubook}.

%We now consider the multilinear case.
For the sake of simplicity, we only consider the bilinear case. Let $m\in C^{s}(\mathbb{R}^{2n}\backslash\{0\})$, for some integer $s$, satisfying the following condition:
\begin{equation}\label{Homandercondition-multipliers}
|\p^\alpha_{\xi}\p^\beta_{\eta}m(\xi,\eta)|\leq C_{\alpha,\beta}(|\xi|+|\eta|)^{-(|\alpha|+|\beta|)}
\end{equation}
for all $|\alpha|+|\beta|\leq s$ and $(\xi,\eta)\in \mathbb{R}^{2n}\backslash\{0\}$. The bilinear Fourier multiplier operator $T_m$ is defined by
$$
T_m(f,g)(x)=\f{1}{(2\pi)^{2n}}\int_{\mathbb{R}^n}\int_{\mathbb{R}^n}e^{ix\cdot (\xi+\eta)}m(\xi,\eta)\hat{f}(\xi)\hat{g}(\eta)d\xi d\eta
$$
for all $f, g\in \mathcal{S}(\mathbb{R}^n)$. \\
The associated kernel $K(x, y_1, y_2)$ to $T_m$ is given by
\begin{equation}\label{associatedkenerl}
K(x,y_1,y_2)= \check{m} (x-y_1,x-y_2)
\end{equation}
where $\check{m}$ is the inverse Fourier transform of $m$. It is proved in \cite{BD} that the associated kernel $K$ satisfies (H2).
\begin{prop}\label{kernelcondtionH2}
For any $p> 2n/s$, we have,
\begin{equation}\label{eq1-kernelH2}
\Big(\int_{S_j(Q)}\int_{S_k(Q)}|K(x,y_1,y_2)-K(\overline{x},y_1,y_2)|^{p'}dy_1dy_2\Big)^{1/p'}\leq C\f{|x-\overline{x}|^{s-2n/{p}}}{|Q|^{s/n}} 2^{-s \max\{j,k\}}
\end{equation}
for all balls $Q$, all $x, \overline{x}\in \f{1}{2}Q$ and $(j,k)\neq (0,0)$.
\end{prop}
It was shown in \cite{CM2} that if (\ref{Homandercondition-multipliers}) holds for $s>4n$ then $T_m$ maps from $L^{p_1}\times L^{p_2}$ into $L^p$ for all $1<p_1, p_2, p<\vc$ so that $1/p_1+1/p_2=1/p$. Then, it was proved in \cite{GT1} that $T_m$ maps boundedly from $L^{p_1}\times L^{p_2}$ into $L^p$ for all $1<p_1, p_2<\vc$ so that $1/p_1+1/p_2=1/p$ provided that (\ref{Homandercondition-multipliers}) holds for $s\geq 2n+1$. However, in the sense of the linear case, the number of derivatives $s\geq 2n+1$  is not optimum and it is natural to expect that we only need $s\geq n+1$. The first positive answer is due to Tomita \cite{T2} who proved that if (\ref{Homandercondition-multipliers}) holds for $s\geq n+1$, then $T_m$ maps from $L^{p_1}\times L^{p_2}$ into $L^p$ for all $2\leq p_1, p_2, p<\vc$ such that $1/p_1+1/p_2=1/p$ and then by using the multilinear interpolation and duality arguments, he obtained that $T_m$ maps from $L^{p_1}\times L^{p_2}$ into $L^p$ for all $1<p_1, p_2, p<\vc
 $ such that $1/p_1+1/p_2=1/p$. This result was then improved in \cite{GS} for $p\leq 1$ by using the $L^r$-based Sobolev space, $1<r\leq 2$. A particular case of \cite[Theorem 1.1]{GS} is the following theorem.

\begin{theorem}\label{thmofGS}
Assume that $(\ref{Homandercondition-multipliers})$ holds for some $n+1\leq s\leq 2n$. Then for any $p_1, p_2$ and $p$ such that $\f{2n}{s}<p_1,p_2<\vc$ and $1/p_1+1/p_2=1/p$, the operator $T_m$ maps from $L^{p_1}\times L^{p_2}$ into $L^p$.
\end{theorem}

We remark that the number  $\f{2n}{s}$ in Theorem \ref{thmofGS} is contained implicitly in the proof of \cite[Theorem 1.1]{GS}.\\

For any $2n/s<p_0$, from Theorem \ref{thmofGS} and Proposition \ref{kernelcondtionH2}, $T_m$ satisfies (H1) and (H2) for $p_0$. Then applying the main Theorem \ref{theoremC} we obtain the following result.

\begin{theorem}\label{thmMTm}
Assume that $(\ref{Homandercondition-multipliers})$ holds for some $n+1\leq s\leq 2n$. Let $2n/s<p_0$ Then for any $p_1, p_2, p$ such that $p_0<p_1,p_2<\vc$,
  $1/p_1+1/p_2=1/p$, and $\vec{\omega}=(w_1,w_2)\in A_{\vec{P}/p_0}$ with $\vec{P}=(p_1,p_2)$. Then
$$
\|T_m(f_1, f_2)\|_{L^p(v_{\vec{\omega}})}\leq C [\vec{w}]_{A_{\vec{P}/p_0}}^{\max\left(1,\tfrac{(p_1/p_0)'}{p},\tfrac{(p_2/p_0)'}{p}\right)} \|f_1\|_{L^{p_1}(w_1)}\|f_2\|_{L^{p_2}(w_2)}.
$$
\end{theorem}

\begin{remark}
Similarly to the linear case in Theorem \ref{thm1-App}, it is natural to raise the question that whether or not the weighted bound in Theorem \ref{thmMTm} holds true for $p_0=2n/s$. This is an open question and would be our future research.
\end{remark}

\bigskip

\textbf{Acknowledgement.} The first and the third named authors were supported by Australian Research Council. The second named author was supported by the ERC StG-256997-CZOSQP.


\begin{thebibliography}{999}
\label{Sect:Bibliography}
%**********************************************************************************************************%
%\bibitem{BS} C. Bennett and R. Sharpley, \emph{Interpolation of Operators}, Academic Press, New York, (1988).

\bibitem{Bu} S. M. Buckley, {Estimates for operator norms on weighted spaces and reverse Jensen inequalities}, Trans. Amer. Math. Soc. {340} (1993), 253-–272.

\bibitem{BD} T. A. Bui and X. T. Duong, Weighted norm inequalities for multilinear operators and applications to multilinear Fourier multipliers, Bull. Sci. Math. 137 (2013), no. 1, 63--75.

\bibitem{CT} A. P. Calder\'on and A. Torchinsky, Parabolic maximal functions associated with a distribution. II, Advances in Math. 24 (1977), 101-171.

%\bibitem{CDM} R. R. Coifman, D. Deng and Y. Meyer, Domains de la racine carr\'ee de certains op\'erateurs diff\'erentiels accr\'etifs, Ann. Inst. Fourier (Grenoble) 33 (1983), 123--134.

%\bibitem{CMM} R .R. Coifman, A. McIntosh and Y. Meyer, L’integrale de Cauchy definit un operateur borne sur $L^2$ pour les courbes lipschitziennes, Ann. of Math. 116 (1982), 361--387.

%\bibitem{CM} R .R. Coifman and Y. Meyer, Au-del\`a des op\'erateurs pseudo-diff\'erentiels, Asterisque 57 (1978).

%\bibitem{CHO} L. Chaffee, J. Hart, and L. Oliveira, Weighted multilinear square functions bounds, Michigan Math. J. 63 (2014), 371–-400.

\bibitem{Chen2013}  W. Chen and W. Damian, Weighted estimates for the multisublinear maximal function, Rend. Circ. Mat. Palermo 62 (2015), 379--391.

%\bibitem{CXY} X. Chen, Q. Xue and K. Yabuta, On multilinear Littlewood–-Paley operators, Nonlinear Anal 115 (2015), 25--40.


%\bibitem{CM} R. Coifman and Y. Meyer, On commutators of singular integral and bilinear singular integrals,
%Trans. Amer. Math. Soc. 212 (1975), 315--331.

\bibitem{CM2} R. Coifman and Y. Meyer, Au del\`a des op\'erateurs pseudodiff\'erentiels, \emph{Ast\'erisque}, 57 (1978).

%\bibitem{CM3} R. Coifman and Y. Meyer, Ondelettes \'et op\'erateurs, III, Hermann, Paris, 1990.

\bibitem{CondeThesis} J. M. Conde-Alonso, Geometric and probabilistic methods in Calder\'on-Zygmund theory, Phd. Thesis, 2015.

\bibitem{CR}J. M. Conde-Alonso and G. Rey, A pointwise estimate for positive dyadic shifts and some applications,  Available at http://arxiv.org/pdf/1409.4351v1.pdf

%\bibitem{CMP} D. Cruz-Uribe, J. M. Martell and C. Per\'ez, Sharp weighted estimates for classical operators, To appear in Adv. Math..

\bibitem{DLP} W. Dami\'an, A. K. Lerner and C. P\'erez, {Sharp weighted bounds for multilinear maximal functions and Calder\'on-Zygmund operators}, J. Fourier Anal. and Appl.  21 (2015), no. 1, 161--181.

%\bibitem{DJ} G. David and J.L. Journe, Une caract\'erisation des op\'erateurs int\'egraux singuliers born\'es sur $L^2(\mathbb{R}^n)$, C. R. Math. Acad. Sci. Paris 296 (1983), 761--764.
%\bibitem{Dubook} J. Duoandikoetxea, Fourier Analysis, Grad. Stud. math, 29, American Math. Soc., Providence,
%2000.

%\bibitem{FJK1} E.B. Fabes, D. Jerison and C. Kenig, Multilinear Littlewood–-Paley estimates with applications to partial differential equations, Proc. Natl. Acad. Sci. 79 (1982), 5746--5750.

%\bibitem{FJK2} E.B. Fabes, D. Jerison and C. Kenig, Necessary and sufficient conditions for absolute continuity of elliptic harmonic measure, Ann. of Math. 119 (1984), 121--141.

%\bibitem{FJK3} E.B. Fabes, D. Jerison and C. Kenig,  Multilinear square functions and partial differential equations, Amer. J. Math. 107 (1985), 1325--1368.

%\bibitem{FT}  M. Fujita and N. Tomita, Naohito, A counterexample to weighted estimates for multilinear Fourier multipliers with Sobolev regularity,
%J. Math. Anal. Appl. 409 (2014), no. 2, 630--636.

\bibitem{GLP} Z. Guo, P. Li and L. Peng, $L^p$ boundedness of commutators of Riesz transforms associated to Schr\"odinger operators, J. Math. Anal. Appl. 341 (2008), 421--432.

 \bibitem{DGPP} O. Dragi\v{c}evi\'c, L. Grafakos, M.C. Pereyra and S. Petermichl, {Extrapolation and sharp norm estimates for classical operators on weighted Lebesgue spaces}, Publ. Math. {49} (2005), 73--91.

   \bibitem{G} L. Grafakos, \emph{Modern Fourier Analysis, 3rd Edition}, GTM 250, Springer, New York,  2014.

\bibitem{GS} L. Grafakos and Z. Si, The H\"ormander type multiplier theorem for multilinear operators, J. Reine Angew. Math. 668 (2012), 133–147.

  \bibitem{GT1} L. Grafakos and R.H. Torres, {Multilinear Calder\'on--Zygmund theory}, Adv. Math. {165} (2002), 124--164.

%\bibitem{H} J. Hart, Bilinear square functions and vector-valued Calder\'on-Zygmund operators, J. Fourier Anal. Appl. 18 (2012), 1291--1313.


\bibitem{Ho} L. H\"ormander, Estimates for stranlation invariant operators in $L^p$ spaces, Acta Math. 104 (1960), 93-139.

  \bibitem{Hyt1}   T. Hyt\"{o}nen, {The sharp weighted bound for general Calder\'on--Zygmund operators}, Ann. of Math. (2) {175} (2012), no. 3, 1473--1506.

  \bibitem{Hyt2} T. Hyt\"onen, {The {$A_2$} theorem: Remarks and complements}, Contemp. Math., 612, Amer. Math. Soc., Providence, RI (2014), 91--106

\bibitem{HLP} T. Hyt\"onen, M. Lacey and C. P\'erez, Sharp weighted bounds for the q-variation of singular integrals. Bull. Lond. Math. Soc. 45 (2013), no. 3, 529–540.

 \bibitem{KW} D. S. Kurtz and R. L. Wheeden, Results on weighted norm inequalities for multipliers, Trans.
Amer. Math. Soc. 255 (1979), 343--362.

  \bibitem{L1}  A. K. Lerner, { A pointwise estimate for the local sharp maximal function with applications to singular integrals}, Bull. London Math. Soc. {42} (2010), no. 5, 843--856.

  \bibitem{Ler3} A. K. Lerner, {On an estimate of Calder\'on-Zygmund operators by dyadic positive operators}, J. Anal. Math. {121} (2013), 141--161.

  \bibitem{Ler4} A. K. Lerner, {A simple proof of the {$A_2$} conjecture}, Int. Math. Res. Not. IMRN {14} (2013), 3159--3170.

%\bibitem{Ler5} A.K.  Lerner, {On sharp aperture-weighted estimates for square functions}, J. Fourier Anal. Appl. 20 (2014), no. 4, 784--800.

% \bibitem{Ler6} A. K. Lerner, Sharp weighted norm inequalities for Littlewood–-Paley operators and singular integrals, Adv. Math. 226 (2011), 3912--3926.
%    \bibitem{Ler7} A.K.  Lerner, {On sharp aperture-weighted estimates for square functions}, J. Fourier Anal. Appl. 20 (2014), no. 4, 784--800.


%\bibitem{LN} A. K. Lerner and F. Nazarov, {\it Intuitive dyadic calculus}. Available at
%http://www.math.kent.edu/ zvavitch/Lerner Nazarov Book.pdf.


\bibitem{LOPTT} A. K. Lerner, S. Ombrosi, C. P\'erez, R. H. Torres and R. Trujillo-Gonz\'alez, New maximal functions and multiple weights for the multilinear Calder\'on-Zygmund theory, Advances in Math. {220} (2009), 1222--1264.


  \bibitem{LMS} K. Li, K. Moen and W. Sun,
  The sharp weighted bound for multilinear
  maximal functions and {C}alder{\'o}n-{Z}ygmund operators,
  to appear in J. Fourier Anal. and Appl.
  Available at \href{http://arxiv.org/abs/1212.1054}{1212.1054}.

\bibitem{LMPR} M. Lorente, J. M. Martell, C. Per\'ez and M. S. Riveros, Generalized H\"ormander's conditions and weighted endpoint estimates, Studia Math. 195 (2009), no. 2, 157--192.

\bibitem{LMRT} M. Lorente, J. M. Martell, M. S. Riveros and A. de la Torre, Generalized H\"ormander's condition,
commutators and weights, J. Math. Anal. Appl. 342 (2008), 1399--1425.

\bibitem{LRT} M. Lorente, M. S. Riveros and A. de la Torre, Weighted estimates for singular integral operators
satisfying H\"ormander's conditions of Young type, J. Fourier Anal. Apl. 11 (2005), 497--509.

\bibitem{Moen} K. Moen, Sharp weighted bounds without testing or extrapolation, Arch. Math. (Basel) {99} (2012), 457--466.

\bibitem{Mu} B. Muckenhoupt, Weighted norm inequalities for the Hardy maximal function, Trans. Amer. Math. J. 19 (1972), 207--226.

 \bibitem{Pet} S. Petermichl, The sharp bound for the {H}ilbert transform on weighted {L}ebesgue spaces in terms of the classical {$A_p$} characteristic, Amer. J. Math. 129 (5) (2007), 1355--1375.

  \bibitem{Pet2} S. Petermichl, The sharp weighted bound for the {R}iesz transforms, Proc. Amer. Math. Soc. 136 (4) (2008), 1237--1249.


\bibitem{Shen} Z. Shen,  $L^p$ estimates for Schr\"odinger operators with certain potentials, Ann. Inst. Fourier (Grenoble) 45(2)  (1995), 513--546.

\bibitem{T2} N. Tomita, A H\"ormander type multiplier theorem for multilinear operators, \emph{Journal of
Functional Analysis} \textbf{259} (2010), 2028-2044.

%\bibitem{T}A. Torchinsky, \emph{Real-Variable Methods in Harmonic Analysis}, Academic Press, New York, 1986.




\end{thebibliography}
\end{document}